\documentclass[10pt]{amsart}
\usepackage[latin1]{inputenc}
\usepackage{amsmath}
\usepackage{amsfonts}
\usepackage{amssymb}
\usepackage{amsthm}
\usepackage{graphicx,epsfig}
\usepackage{amscd}
\DeclareMathOperator{\isomzero}{Isom_0}
\DeclareMathOperator{\isom}{Isom}

\newcommand{\M}{\mathbb{M}}
\newcommand{\R}{\mathbb{R}}
\newcommand{\h}{\mathbb{H}}
\newcommand{\Z}{\mathbb{Z}}

\newcommand{\C}{\mathbb{C}}
\newcommand{\s}{\mathbb{S}}

\newcommand{\PP}{\mathbb{P}}

\newcommand{\rmT}{\mathrm{T}}
\newcommand{\rmS}{\mathrm{S}}
\newcommand{\rmd}{\mathrm{d}}

\newcommand{\rmR}{\mathrm{R}}

\newcommand{\rmL}{\mathrm{L}}

\newcommand{\cS}{{\mathcal S}}

\newcommand{\cC}{{\mathcal C}}

\newcommand{\cX}{{\mathcal X}}

\let\8=\infty

\numberwithin{equation}{section}

\begin{document}

\newtheorem{thm}{Theorem}[section]
\newtheorem*{thmintro}{Theorem}
\newtheorem{cor}[thm]{Corollary}
\newtheorem{prop}[thm]{Proposition}
\newtheorem{app}[thm]{Application}
\newtheorem{lemma}[thm]{Lemma}
\newtheorem{notation}[thm]{Notations}
\newtheorem{hypothesis}[thm]{Hypothesis}

\newtheorem{defin}[thm]{Definition}
\newenvironment{defn}{\begin{defin} \rm}{\end{defin}}
\newtheorem{remk}[thm]{Remark}
\newenvironment{rem}{\begin{remk} \rm}{\end{remk}}
\newtheorem{exa}[thm]{Example}
\newenvironment{ex}{\begin{exa} \rm}{\end{exa}}
\newtheorem{cla}[thm]{Claim}
\newenvironment{claim}{\begin{cla} \rm}{\end{cla}}

\title{Minimal isometric immersions into $\s^2\times\R$ and $\h^2\times\R$}

\author{Beno\^\i t Daniel}
\address{Universit\'e de Lorraine\\
Institut \'Elie Cartan de Lorraine\\
UMR 7502\\
CNRS\\
B.P. 70239\\
F-54506 Vand\oe{}uvre-l\`es-Nancy cedex\\
FRANCE}
\email{benoit.daniel@univ-lorraine.fr}

\thanks{The author was partially supported by the ANR-11-IS01-0002 grant.}

\date{}



\begin{abstract}
For a given simply connected Riemannian surface $\Sigma$, we relate the problem of finding minimal isometric immersions of $\Sigma$ into $\s^2\times\R$ or $\h^2\times\R$ to a system of two partial differential equations on $\Sigma$. We prove that a constant intrinsic curvature minimal surface in $\s^2\times\R$ or $\h^2\times\R$ is either totally geodesic or part of an associate surface of a certain limit of catenoids in $\h^2\times\R$. We also prove that if a non constant curvature Riemannian surface admits a continuous one-parameter family of minimal isometric immersions into $\s^2\times\R$ or $\h^2\times\R$, then all these immersions are associate.
\end{abstract}

\maketitle

\section{Introduction} \label{sec:intro}

It is a classical result that any simply connected minimal surface in Euclidean space $\R^3$ admits a one-parameter family of minimal isometric deformations, called the \emph{associate family}. Conversely, two minimal isometric immersions of the same Riemannian surface into $\R^3$ are associate. These are easy consequences of the Gauss and Codazzi equations in $\R^3$. More generally, analogous results hold for constant mean curvature (CMC) surfaces in $3$-dimensional space forms.

The aim of this paper is to investigate extensions of these results and  related questions for minimal surfaces in the product manifolds $\s^2\times\R$ and  $\h^2\times\R$, where $\s^2$ is the $2$-sphere of curvature $1$ and $\h^2$ is the hyperbolic plane of curvature $-1$.


The systematic study of minimal surfaces in $\s^2\times\R$ and  $\h^2\times\R$ was initiated by H. Rosenberg and W. Meeks \cite{rosenillinois,mrcmh} and has been very active since then. The existence of an associate family for simply connected minimal surfaces in $\s^2\times\R$ and  $\h^2\times\R$ was proved independently by the author \cite{danieltams} and by L. Hauswirth, R. Sa Earp and E. Toubiana \cite{hset}. 

On the other hand, there exist examples of isometric minimal surfaces in $\h^2\times\R$ that are not associate. For instance, a certain limit of rotational catenoids has curvature $-1$ (see Example \ref{ex:catenoid} for details); we will call it the \emph{parabolic generalized catenoid}. This surface is consequently isometric to a horizontal hyperbolic plane $\h^2\times\{a\}$, but these two surfaces are not associate (actually this example provides a two-parameter family of non associate minimal isometric immersions). The parabolic generalized catenoid and the helicoid of Example 18 in \cite{hset} are associate. 
Also, R. Sa Earp \cite{saearp} gave examples of pairs of non associate isometric minimal surfaces in $\h^2\times\R$ that are invariant by hyperbolic screw-motions (see Example \ref{ex:saearp} for details).   

These examples show that the classical result in $\R^3$ cannot be extended to  $\h^2\times\R$. Hence it is a natural problem to investigate classifications of non associate isometric minimal surfaces (for instance this is proposed in \cite{gmmcag}, p. 932). The main result of this paper is a rigidity result: we prove that if a minimal surface in $\s^2\times\R$ or  $\h^2\times\R$ admits a one-parameter family of minimal isometric deformations, then this family is the associate family, unless the surface is a horizontal hyperbolic plane or an associate surface of a parabolic generalized catenoid in  $\h^2\times\R$ (Corollary \ref{family} and Remark \ref{rem:family}).

In this paper we consider a real constant $c\neq0$ and we let $\M^2(c)$ denote the simply connected Riemannian surface of constant curvature $c$. In particular $\s^2=\M^2(1)$ and $\h^2=\M^2(-1)$. By scaling, it is not restrictive to assume that $c\in\{1,-1\}$; however we will generally not do this normalization, except when dealing with previously known examples of minimal surfaces that were computed with this normalization.

We will use the compatibility equations for surfaces in $\M^2(c)\times\R$ obtained in \cite{danieltams}. The \emph{angle function}, i.e., the inner product of the unit normal with the unit upward pointing vertical field, plays an important role in these equations.
The strategy is to study the angle function $\nu$ of a minimal isometric immersion and to reduce the compatibility equations to a system of two partial differential equations satisfied by $\nu$. This system is similar to the one satisfied by the K\"ahler angle of the surfaces in K\"ahler-Einstein $4$-manifolds studied in \cite{km,tutams,tu2}, but in these papers the surfaces that are considered have constant intrinsic curvature, which implies the constancy of the solution or the reduction to a system of ordinary differential equations. In our case we also treat non constant curvature surfaces, for which further equations will be needed. The main results will follow from the study of this system.

The paper is organized as follows. In Section \ref{sec:angle}, we recall some results about isometric immersions into $\M^2(c)\times\R$ and we establish the system satisfied by the angle function (Theorem \ref{thm:system}). In Section \ref{sec:resolution} we derive some new equations from this system; these are compatibility equations. In particular, when the curvature is not constant, we obtain a pointwise polynomial order $0$ equation (Proposition \ref{propP}).

Section \ref{sec:constantK} is devoted to the study of constant intrinsic curvature minimal surfaces. We classify them (even locally) in Theorem \ref{thm:cc}. We think this theorem is particularily interesting in $\h^2\times\R$, since we have a classification of constant curvature minimal $2$-dimensional submanifolds of a $3$-dimensional manifold where non totally geodesic examples appear. 

Finally, in Section \ref{sec:number} we prove that, up to congruences, there cannot exist more than $6$ minimal isometric immersions of a non constant curvature surface such that no two of them are associate (Theorem \ref{thm:number}). In particular, if a non constant curvature surface admits a continuous one-parameter family of minimal isometric immersions, then all these immersions are associate (Corollary \ref{family}).


\section{The angle function} \label{sec:angle}

\subsection{The compatibility equations and the associate family}

We first fix some notation and recall some definitions and results. The projection $\M^2(c)\times\R\to\R$ is called the \emph{height function}. We let $\xi$ denote the upward pointing unit vector field that is tangent to the factor $\R$, that is, $\xi$ is the gradient of the height function.

We let $\isom(\M^2(c)\times\R)$ denote the isometry group of $\M^2(c)\times\R$. It has $4$ connected components. The connected component of the identity consists of isometries that preserve the orientations of both $\M^2(c)$ and $\R$; we will denote it by $\isomzero(\M^2(c)\times\R)$.
We say that two immersions $f:\Sigma\to\M^2(c)\times\R$ and $g:\Sigma\to\M^2(c)\times\R$ are \emph{congruent} if there exists $\Phi\in\isom(\M^2(c)\times\R)$ such that $g=\Phi\circ f$.

We will make use of the following theorem for local isometric immersions into $\M^2(c)\times\R$.

\begin{thm}[\cite{danieltams}] \label{thm:tams}
Let $(\Sigma,\rmd s^2)$ be an oriented simply connected Riemannian surface. Let $K$ be the curvature of $\rmd s^2$. Let $S:\rmT\Sigma\to\rmT\Sigma$ be a field of symmetric operators, $T\in\cX(\Sigma)$ and $\nu:\Sigma\to[-1,1]$ be a smooth function. Then there exists and isometric immersion $f:\Sigma\to\M^2(c)\times\R$ such that the shape operator with respect to the normal $N$ associated to $f$ is
$$\rmd f\circ\rmS\circ\rmd f^{-1}$$ and such that
$$\xi=\rmd f(T)+\nu N$$ if and only if the $4$-tuple $(\rmd s^2,S,T,\nu)$ satisfies the following equations on $\Sigma$:
\begin{equation} \label{eq:gauss}
K=\det S+ c\nu^2,
 \tag{C1}
\end{equation}
\begin{equation} \label{eq:codazzi}
\nabla_X SY-\nabla_Y SX-S[X,Y]=
 c\nu(\langle Y,T\rangle X-\langle X,T\rangle Y),
\tag{C2}
\end{equation}
\begin{equation} \label{eq:T1}
\nabla_XT=\nu SX,
\tag{C3}
\end{equation}
\begin{equation} \label{eq:T2}
\rmd\nu(X)+\langle SX,T\rangle=0.
\tag{C4}
\end{equation}
\begin{equation} \label{eq:normT}
||T||^2+\nu^2=1.
\tag{C5}
\end{equation}

If this is the case, then the immersion is moreover unique up to an isometry in $\isomzero(\M^2(c)\times\R)$.
\end{thm}

We will refer to equations \eqref{eq:gauss}, \eqref{eq:codazzi}, \eqref{eq:T1}, \eqref{eq:T2}, \eqref{eq:normT} as the \emph{compatibility equations} for surfaces in $\M^2(c)\times\R$ and to the $4$-tuple $(\rmd s^2,S,T,\nu)$ as the \emph{Gauss-Codazzi data} of the immersion. The function $\nu$ is called the \emph{angle function} of the immersion. Equations \eqref{eq:gauss} and \eqref{eq:codazzi} are the Gauss and Codazzi equation.

Note that \eqref{eq:gauss}, \eqref{eq:codazzi}, \eqref{eq:T1}, \eqref{eq:T2}, \eqref{eq:normT} are necessary conditions even if $\Sigma$ is not simply connected. Also, if $f:\Sigma\to\M^2(c)\times\R$ is an immersion and $\Phi\in\isomzero(\M^2(c)\times\R)$, then $f$ and  $\Phi\circ f$ have the same Gauss-Codazzi data.

It follows from this theorem (see \cite{danieltams}) that if $\Sigma$ is simply connected and oriented and  $f:\Sigma\to\M^2(c)\times\R$ is a minimal isometric immersion with Gauss-Codazzi data $(\rmd s^2,S,T,\nu)$, then for every $\theta\in\R/(2\pi\Z)$ there exists minimal isometric immersion $f^\theta:\Sigma\to\M^2(c)\times\R$ with Gauss-Codazzi data $(\rmd s^2,e^{\theta J}S,e^{\theta J}T,\nu)$, where $J$ is the rotation of angle $\pi/2$ on $\rmT\Sigma$. Moreover, for each $\theta\in\R/(2\pi\Z)$, $f^\theta$ is unique up to an isometry in $\isomzero(\M^2(c)\times\R)$, and is generically not congruent to $f$ when $\theta\neq0$. The family $(f^\theta)_{\theta\in\R/(2\pi\Z)}$ is called the \emph{associate family} of the immersion $f$. Note that $f^\pi=\sigma\circ f$ where $\sigma\in\isom(\M^2(c)\times\R)$ is the reflection with respect to a horizontal totally geodesic surface $\M^2(c)\times\{a\}$ for some $a\in\R$ (see Proposition 3.8 in \cite{danieltams}). 

Similarly, two minimal surfaces in $\M^2(c)\times\R$ are said to be associate if they are images of two associate minimal isometric immersions of the same surface. There is a slight difference between these two notions; for instance, the parabolic generalized catenoid is the image of non associate minimal isometric immersion of the hyperbolic plane (see Example \ref{ex:catenoid} for details).

\begin{rem}
The existence of the associate family was also proved by L. Haus-wirth, R. Sa Earp and E. Toubiana \cite{hset} using the harmonicity of the horizontal and vertical projections of conformal minimal immersions. We also mention that 
L. Hauswirth and H. Rosenberg \cite{hrmatcontemp} developped the theory of complete finite total curvature minimal surfaces in $\h^2\times\R$ using the relation between the angle function and solutions to the elliptic $\sinh$-Gordon equation. 
\end{rem}

\begin{rem}
The associate family also exists for instance for minimal surfaces in $\C\PP^2$ \cite{egt}. A general discussion about the existence of an associate family can be found in \cite{lodovicipiccione}. 
\end{rem}

\subsection{A system of two partial differential equations}

We first show that, in the case of minimal isometric immersions, the compatibility equations reduce, away from points where $\nu^2=1$, to a system of two partial differential equations involving only the metric $\rmd s^2$ and the angle function $\nu$. Since minimal surfaces in $\M^2(c)\times\R$ are locally graphs of functions satisfying an elliptic partial differential equation with real analytic coefficients, all smooth surfaces, metrics and functions that we consider in this paper will be real analytic.

\begin{thm} \label{thm:system}
Let $\Sigma$ be a minimal surface in $\M^2( c)\times\R$. Then its angle function $\nu:\Sigma\to[-1,1]$ satisfies
\begin{equation} \label{eq:gaussnu}
||\nabla\nu||^2=-(1-\nu^2)(K- c\nu^2), \tag{M1}
\end{equation}
\begin{equation} \label{eq:jacobinu}
\Delta\nu-2K\nu+ c(1+\nu^2)\nu=0, \tag{M2}
\end{equation}
where $K$ denotes the intrinsic curvature of $\Sigma$.

Conversely, let $\Sigma$ be a real analytic simply connected Riemannian surface and $\nu:\Sigma\to(-1,1)$ a smooth function satisfying \eqref{eq:gaussnu} and \eqref{eq:jacobinu} where $K$ is the curvature of $\Sigma$. Then there exists an isometric minimal immersion $f:\Sigma\to\M^2( c)\times\R$ whose angle function is $\nu$. Moreover, if $g:\Sigma\to\M^2( c)\times\R$ is another isometric minimal immersion whose angle function is $\nu$, then $f$ and $g$ are associate.
\end{thm}

\begin{proof}
Let $\Sigma$ be a minimal surface in $\M^2( c)\times\R$ and $(\rmd s^2,S,T,\nu)$ its Gauss-Codazzi data. Considering the orientable double cover if necessary, we may assume that $\Sigma$ is oriented. As we already mentioned, $(\Sigma,\rmd s^2)$ is real analytic. Equation \eqref{eq:T2} and the symmetry of $S$ imply that $ST=-\nabla\nu$; moreover, since $\Sigma$ is minimal, one has $SJ=-JS$ where $J$ denotes the rotation of angle $\pi/2$ in $\rmT\Sigma$. Hence, considering the orthonormal frame $(T/||T||,JT/||T||)$ at a point where $T\neq0$, we obtain that $\det S=-||\nabla\nu||^2/||T||^2$. Then the Gauss equation \eqref{eq:gauss} and equation \eqref{eq:normT} give \eqref{eq:gaussnu} when $T\neq0$, i.e., $\nu^2\neq1$. At points where $\nu^2=1$, one has $\nabla\nu=0$, so \eqref{eq:gaussnu} also holds. 

Since $\xi$ is a Killing field, $\nu$ satisfies $\rmL\nu=0$ where $\rmL$ is the Jacobi operator of $\Sigma$. Since $\rmL=\Delta-2K+ c(1+\nu^2)$ 
(see \cite{danielcmh}, Section 5.2), this gives \eqref{eq:jacobinu}, which concludes the proof of the first assertion. It is also useful to notice that, by \eqref{eq:T1}, $T$ satisfies
\begin{equation} \label{eq:Tgradnu}
\nabla_XT=\frac\nu{1-\nu^2}(-\langle\nabla\nu,X\rangle T+\langle J\nabla\nu,X\rangle JT)
\end{equation}
for every vector field $X$ at points where $\nu^2\neq1$.

Let now  $\Sigma$ be a real analytic simply connected Riemannian surface and $\nu:\Sigma\to(-1,1)$ a smooth function satisfying \eqref{eq:gaussnu} and \eqref{eq:jacobinu}. The fact that $\nu$  satisfies \eqref{eq:jacobinu} implies that it is real analytic. The first step is to find a vector field $T$ satisfying \eqref{eq:normT} and \eqref{eq:Tgradnu}. Let $(e_1,e_2)$ be an orthonormal frame defined in an open set $U\subset\Sigma$, and let $(\omega_1,\omega_2)$ be its dual coframe. We assume that $U$ is simply connected. Let $\alpha$ be the $1$-form on $U$ such that $\nabla_Xe_1=\alpha(X)e_2$ for every $X\in\cX(U)$. Let $\theta:U\to\R$ be a smooth function and $T=\sqrt{1-\nu^2}e^{\theta J}e_1$. Then, as $\nu^2\neq1$,
$$\nabla_XT=-\frac{\nu}{1-\nu^2}\rmd\nu(X)T+\rmd\theta(X)JT+\alpha(X)JT,$$ so $T$ satisfies \eqref{eq:Tgradnu} if and only if
$$\frac\nu{1-\nu^2}\langle J\nabla\nu,X\rangle=\rmd\theta(X)+\alpha(X)$$ for every $X\in\cX(U)$, i.e., if and only if
\begin{equation} \label{eq:dtheta}
\rmd\theta=-\frac\nu{1-\nu^2}(\rmd\nu\circ J)-\alpha.
\end{equation}
We have
\begin{eqnarray*}
\rmd\left(\frac\nu{1-\nu^2}(\rmd\nu\circ J)+\alpha\right) & = & 
\frac{1+\nu^2}{(1-\nu^2)^2}\rmd\nu\wedge(\rmd\nu\circ J)+\frac\nu{1-\nu^2}\rmd(\rmd\nu\circ J)+\rmd\alpha \\
& = & \left(-\frac{1+\nu^2}{(1-\nu^2)^2}||\nabla\nu||^2-\frac\nu{1-\nu^2}\Delta\nu-K\right)\omega_1\wedge\omega_2 \\
& = & 0
\end{eqnarray*}
by \eqref{eq:gaussnu} and \eqref{eq:jacobinu}. So \eqref{eq:dtheta} has a solution $\theta:U\to\R$, which is unique up to an additive constant $\theta_0$. Hence there exists $T\in\cX(U)$ satisfying \eqref{eq:normT} and \eqref{eq:Tgradnu}, and $T$ is unique up to the rotation by a fixed angle $\theta_0$. We now fix such a vector field $T$. Since $T$ does not vanish on $U$, there exists a unique symmetric traceless operator $S:\rmT U\to\rmT U$ such that $ST=-\nabla\nu$.

We now prove that $(\rmd s^2,S,T,\nu)$ satisfies \eqref{eq:gauss}, \eqref{eq:codazzi}, \eqref{eq:T1}, \eqref{eq:T2} and \eqref{eq:normT} on $U$. Equations \eqref{eq:normT} and \eqref{eq:T2} follow immediately from the definitions of $T$ and $S$ and from the symmetry of $S$, and then \eqref{eq:T1} follows from \eqref{eq:Tgradnu}. Following a previous computation we also get $\det S=-||\nabla\nu||^2/||T||^2$, so \eqref{eq:gaussnu} gives \eqref{eq:gauss}. 

It now suffices to check that \eqref{eq:codazzi} is satisfied for $X=T$ and $Y=JT$. First, using \eqref{eq:T1} and the fact that $SJ=-JS$ and that $\nabla_Z$ commutes with $J$ for every $Z$, we get $$[T,JT]=2\nu JST,$$
so, since $S$ is symmetric, 
$$\langle S[T,JT],T\rangle=0,\quad
\langle S[T,JT],JT\rangle=-2\nu||ST||^2=-2\nu||\nabla\nu||^2.$$
Also, $$\nabla_TSJT-\nabla_{JT}ST=J\nabla_T\nabla\nu+\nabla_{JT}\nabla\nu,$$ which yields (using the symmetry of the Hessian $(X,Y)\mapsto\langle X,\nabla_Y\nabla\nu\rangle$)
$$\langle\nabla_TSJT-\nabla_{JT}ST,T\rangle=0, \quad
\langle\nabla_TSJT-\nabla_{JT}ST,JT\rangle=||T||^2\Delta\nu.$$
Consequently we have
$$\langle\nabla_TSJT-\nabla_{JT}ST-S[T,JT],T\rangle=0,$$
and, by \eqref{eq:gaussnu}, \eqref{eq:jacobinu} and \eqref{eq:normT},
$$\langle\nabla_TSJT-\nabla_{JT}ST-S[T,JT],JT\rangle=-c\nu(1-\nu^2)^2.$$
These two equations are \eqref{eq:codazzi} for $X=T$ and $Y=JT$. This concludes the proof of the fact that $(\rmd s^2,S,T,\nu)$ satisfies \eqref{eq:gauss}, \eqref{eq:codazzi}, \eqref{eq:T1}, \eqref{eq:T2} and \eqref{eq:normT} on $U$.

Since $\Sigma$ is simply connected, classical arguments prove that we can extend $S$ and $T$, defined on $U$, to the whole $\Sigma$ in a unique way. Since $T$ is unique on $U$ up to a rotation  by a fixed angle $\theta_0$ and $S$ is defined uniquely in terms of $T$ and $\nu$, this concludes the proof.
\end{proof}

In the second part of this theorem, we do not treat the case of functions $\nu$ that take values $1$ and $-1$, but we will not need this in the sequel.

\begin{rem} \label{minusnu} 
 It is clear that $\nu$ satisfies \eqref{eq:gaussnu} and \eqref{eq:jacobinu} if and only if $-\nu$ does. The associate families of immersions having $\nu$ and $-\nu$ as angle functions differ by a rotation of angle $\pi$ around a horizontal geodesic in $\M^2(c)\times\R$ (see Proposition 3.8 in \cite{danieltams}). 
\end{rem}

\begin{rem}
 L. Hauswirth, R. Sa Earp and E. Toubiana \cite{hset} studied minimal immersions using the fact that the projections into $\h^2$ and into $\R$ are harmonic. In particular they proved that two minimal isometric immersions whose height functions have the same Hopf differential are congruent.
\end{rem}

From Theorem \ref{thm:system} follows the next result, which was already noted in \cite{dillen1,dillen2} (where, more generally, all surfaces in $\M^2(c)\times\R$ with constant angle function were classified).

\begin{lemma} \label{constantangle}
 Let $\Sigma$ be a minimal surface in $\M^2( c)\times\R$ with constant angle function $\nu$. Then
\begin{itemize}
 \item either $\nu^2=1$, $K= c$ and $\Sigma$ is part of a horizontal surface $\M^2(c)\times\{a\}$ for some $a\in\R$,
\item either $\nu=0$, $K=0$ and $\Sigma$ is part of a vertical surface $\gamma\times\R$ where $\gamma$ is a geodesic of $\M^2(c)$.
\end{itemize}
In particular, $\Sigma$ is totally geodesic.
\end{lemma}

\begin{proof}
Assume that $\nu^2=1$. Then by \eqref{eq:jacobinu} we have $K=c$. Moreover, with the notation of Theorem \ref{thm:tams}, we have $T=0$, so the height function of the surface is constant, i.e., $\Sigma$ is part of a horizontal surface $\M^2( c)\times\{a\}$ for some $a\in\R$.

Assume that $\nu^2<1$. Then by \eqref{eq:gaussnu} we have $K=c\nu^2$. Reporting in \eqref{eq:jacobinu} yields $c\nu(1-\nu^2)=0$, so $\nu=0$ and $K=0$. The fact that $\nu=0$ means that the vertical field $\xi$ is tangent to $\Sigma$ everywhere, so there exists a curve $\gamma\subset\M^2(c)$ such that $\Sigma$ is part of $\gamma\times\R$. Since $\Sigma$ is minimal, $\gamma$ is a geodesic of $\M^2(c)$.
\end{proof}

Note that conversely it is clear from \eqref{eq:codazzi} and \eqref{eq:normT} that a totally geodesic surface has either $\nu=0$ or $T=0$, i.e., $\nu^2=1$.

\section{Study of the system} \label{sec:resolution}

In this section we derive some further equations from \eqref{eq:gaussnu} and \eqref{eq:jacobinu} that will be useful to answer some geometric questions. Here $\Sigma$ is a real analytic Riemannian surface, $\nabla$ its Riemannian connection, $\rmR$ its Riemann curvature tensor with the following sign convention:
$$\rmR(X,Y)Z=\nabla_Y\nabla_XZ-\nabla_X\nabla_YZ+\nabla_{[X,Y]}Z,$$
and $K$ its curvature.

We first settle some notation. If $f$ is a smooth function on $\Sigma$, its Hessian $\nabla^2f$ is defined by $$(\nabla^2f)(X,Y)=\langle X,\nabla_Y\nabla f\rangle.$$ This is a symmetric $2$-tensor. The Laplace-Beltrami operator $\Delta$ applied to $f$ is the trace of $\nabla^2f$. 

We now consider a local orthonormal frame $(e_1,e_2)$. Let $J$ be the linear operator such that $Je_1=e_2$ and $Je_2=-e_1$. There exists a $1$-form $\alpha$ such that $$\nabla_Xe_i=\alpha(X)Je_i$$ for all vector fields $X$. Setting $\alpha_i=\alpha(e_i)$, we have
$$\nabla_{e_1}e_1=\alpha_1e_2,\quad
\nabla_{e_2}e_1=\alpha_2e_2,\quad
\nabla_{e_1}e_2=-\alpha_1e_1,\quad
\nabla_{e_2}e_2=-\alpha_2e_1.$$
We will set 
$$f_i=\langle e_i,\nabla f\rangle=e_i\cdot f,$$
$$f_{ij}=(\nabla^2f)(e_i,e_j)=e_j\cdot f_i-(\nabla_{e_j}e_i)\cdot f.$$
In the sequel we will use the following differentation formulas:
$$\begin{array}{ll}
e_1\cdot f_1=f_{11}+\alpha_1f_2, &
e_2\cdot f_1=f_{12}+\alpha_2f_2, \\
e_1\cdot f_2=f_{12}-\alpha_1f_1, &
e_2\cdot f_2=f_{22}-\alpha_2f_1.
\end{array}$$

Let $\nu:\Sigma\to[-1,1]$ be a smooth (hence real analytic) function satisfying \eqref{eq:gaussnu} and \eqref{eq:jacobinu}. In the frame $(e_1,e_2)$, equations \eqref{eq:gaussnu} and \eqref{eq:jacobinu} read as
\begin{equation}\label{E1-1}
 \nu_1^2+\nu_2^2=-(1-\nu^2)(K- c \nu^2),\tag{E1}
\end{equation}
\begin{equation}\label{E2-1}
 \nu_{11}+\nu_{22}=\nu(2K- c(1+\nu^2)).\tag{E2-1}
\end{equation}

The first step is to obtain another order one equation.

\begin{lemma}
If $\nabla\nu$ does not vanish and $\nu$ satisfies \eqref{E1-1} and \eqref{E2-1}, then $\nu$ satisfies 
\begin{equation} \label{E2-2}
 2(K- c \nu^2)\nu_{12}=K_1\nu_2+K_2\nu_1-6 c\nu\nu_1\nu_2, \tag{E2-2}
\end{equation}
\begin{equation} \label{E2-3}
(K- c \nu^2)(\nu_{11}-\nu_{22})=-3 c\nu(\nu_1^2-\nu_2^2)+K_1\nu_1-K_2\nu_2.
 \tag{E2-3}
\end{equation}
\end{lemma}

\begin{proof}
By \eqref{E1-1}, the fact that $\nabla\nu$ does not vanish implies that $1-\nu^2$ and $K- c \nu^2$ do not vanish. 

Differentiating \eqref{E1-1}  yields
\begin{equation} \label{eq:d1e1}
2(\nu_1\nu_{11}+\nu_2\nu_{12})=2(K+ c(1-2\nu^2))\nu\nu_1-(1-\nu^2)K_1,
\end{equation}
 \begin{equation} \label{eq:d2e1}
2(\nu_1\nu_{12}+\nu_2\nu_{22})=2(K+ c(1-2\nu^2))\nu\nu_2-(1-\nu^2)K_2.
\end{equation}
Then, reporting \eqref{E1-1} and \eqref{E2-1} in $\nu_2$\eqref{eq:d1e1}$+\nu_1$\eqref{eq:d2e1} gives \eqref{E2-2}.

Also, $\nu_1$\eqref{eq:d1e1}$-\nu_2$\eqref{eq:d2e1} gives
$$2(\nu_1^2\nu_{11}-\nu_2^2\nu_{22})
=2(K+ c(1-2\nu^2))\nu(\nu_1^2-\nu_2^2)-(1-\nu^2)(K_1\nu_1-K_2\nu_2).$$
Then, using the fact that
$$2(\nu_1^2\nu_{11}-\nu_2^2\nu_{22})=(\nu_{11}+\nu_{22})(\nu_1^2-\nu_2^2)+(\nu_{11}-\nu_{22})(\nu_1^2+\nu_2^2)$$ and
 using \eqref{E1-1} and \eqref{E2-1} we obtain \eqref{E2-3}.
\end{proof}

 \begin{prop}
  Let $\nu:\Sigma\to[-1,1]$ be a real analytic function satisfying \eqref{eq:gaussnu} and \eqref{eq:jacobinu}. Then $\nu$ satisfies
\begin{equation} \label{eq:third}
6c\nu\langle\nabla\nu,\nabla K\rangle=||\nabla K||^2-(K-c\nu^2)\Delta K
+4(K-c)(K-c\nu^2)(K+2c\nu^2). 
\tag{M3}
\end{equation} 
 \end{prop}

\begin{proof}
If $\nu$ is constant on an non empty open set, then by analyticity $\nu$ is constant. Then, by Lemma \ref{constantangle}, $K$ is also constant, and either $(K,\nu)=(c,1)$ or $(K,\nu)=(0,0)$. In both cases \eqref{eq:third} is satisfied.

From now on we assume that $\nu$ is not a constant function. By analyticity it sufficies to prove \eqref{eq:third} on a non empty open set $U\subset\Sigma$ on which $\nabla\nu$ does not vanish. By restricting $U$ if necessary, we may also assume that there exists an orthonormal frame $(e_1,e_2)$ on $U$. We now use the previous notations.

We will use the classical Bochner-Weitzenb\"ock formula (see for instance \cite{petersen}, Section 7.3, p. 175):
\begin{equation} \label{eq:bochner}
 \frac12\Delta||\nabla\nu||^2=\langle\nabla\nu,\nabla\Delta\nu\rangle+||\nabla^2\nu||^2+K||\nabla\nu||^2.
\end{equation}
We have $$||\nabla^2\nu||^2=\nu_{11}^2+\nu_{22}^2+2\nu_{12}^2
=\frac12(\nu_{11}+\nu_{22})^2+\frac12(\nu_{11}-\nu_{22})^2+2\nu_{12}^2,$$ so by  \eqref{E2-2} and \eqref{E2-3} we get
\begin{eqnarray*}
2(K-c\nu^2)^2||\nabla^2\nu||^2 & = & (K-c\nu^2)^2(\Delta\nu)^2+9c^2\nu^2||\nabla\nu||^4 \\
& & +||\nabla K||^2||\nabla\nu||^2-6c\nu\langle\nabla K,\nabla\nu\rangle||\nabla\nu||^2. 
\end{eqnarray*}
Also, by \eqref{eq:gaussnu} and \eqref{eq:jacobinu} we have
$$\Delta||\nabla\nu||^2=-(1-\nu^2)\Delta K+4\nu\langle\nabla K,\nabla\nu\rangle
+2(K+c-2c\nu^2)\nu\Delta\nu+2(K+c-6c\nu^2)||\nabla\nu||^2,$$
$$\langle\nabla\nu,\nabla\Delta\nu\rangle=2\nu\langle\nabla K,\nabla\nu\rangle+(2K-c-3c\nu^2)||\nabla\nu||^2.$$
Reporting these three identities into \eqref{eq:bochner} multiplied by $2(K-c\nu^2)^2$ gives 
\begin{eqnarray*}
 0 & = & (K-c\nu^2)^2(1-\nu^2)\Delta K+(K-c\nu^2)^2(-2K\nu-2c\nu+4c\nu^3+\Delta\nu)\Delta\nu \\
& & +9c^2\nu^2||\nabla\nu||^4+||\nabla K||^2||\nabla\nu||^2-6c\nu\langle\nabla K,\nabla\nu\rangle||\nabla\nu||^2 \\
& & +2(K-c\nu^2)^2(2K-2c+3c\nu^2)||\nabla\nu||^2.
\end{eqnarray*}
Dividing by $||\nabla\nu||^2$, taking \eqref{eq:gaussnu} and \eqref{eq:jacobinu} into account, gives
\begin{eqnarray*}
 0 & = & -(K-c\nu^2)\Delta K+3c\nu(K-c\nu^2)\Delta\nu+9c^2\nu^2||\nabla\nu||^2 \\
& & +||\nabla K||^2-6c\nu\langle\nabla K,\nabla\nu\rangle+2(K-c\nu^2)^2(2K-2c+3c\nu^2).
\end{eqnarray*}
We finally obtain \eqref{eq:third} after reporting again \eqref{eq:gaussnu} and \eqref{eq:jacobinu}.
\end{proof}

Equation \eqref{eq:third} is enough to treat the case where $K$ is constant; this will be done in Section \ref{sec:constantK}. The remainder of this section will be devoted to the case where $K$ is not constant. We next derive an order zero equation.

\begin{prop} \label{propP}
Assume that $K$ is not constant. Let $\nu:\Sigma\to[-1,1]$ be a real analytic function satisfying \eqref{eq:gaussnu} and \eqref{eq:jacobinu}. Then there exists a map $P:\Sigma\times\R\to\R$, depending only on the metric, such that
\begin{itemize}
 \item for every point $x\in\Sigma$, the map $P(x,\cdot)$ is an even polynomial map of degree at most $12$,
\item the map $x\mapsto P(x,\cdot)$ is real analytic, 
\item for every point $x\in\Sigma$, 
\begin{equation} \label{eq:fourth}
 P(x,\nu(x))=0,
\tag{M4}
\end{equation}
\item the set $\Sigma\setminus Z$ is dense, where $$Z=\{x\in\Sigma\mid P(x,\cdot)\equiv0\}.$$
\end{itemize}
\end{prop}

\begin{proof}
We first assume that $\Sigma$ is orientable and we choose an orientation on $\Sigma$. We let $J$ denote the rotation of angle $\pi/2$ in $\rmT\Sigma$. Let $U=\{x\in\Sigma\mid\nabla K(x)\neq0\}$ and $V=\{x\in\Sigma\mid\nabla\nu(x)\neq0\}$. Since $K$ is real analytic, $U$ is a dense open subset. By Lemma \ref{constantangle} and since $\nu$ is real analytic, $V$ is also a dense open subset.

On $U$ we consider the orthonormal frame $(e_1,e_2)=(\nabla K/||\nabla K||,J\nabla K/||\nabla K||)$. In this frame we have $K_1=||\nabla K||$ and $K_2=0$. This implies in particular that
\begin{equation} \label{eq:alpha1alpha2}
 K_{12}=\alpha_1K_1,\quad K_{22}=\alpha_2K_1.
\end{equation}

We now do some computations in $U\cap V$ using the frame $(e_1,e_2)$. We have $K-c\nu^2\neq0$ by \eqref{eq:gaussnu} and the definition of $V$. Equation \eqref{eq:third} becomes 
\begin{equation} \label{eq:e12var}
6c\nu K_1\nu_1=A
\end{equation}
with $$A=K_1^2-(K- c \nu^2)\Delta K+4(K- c)(K- c \nu^2)(K+2 c \nu^2).$$
Differentiating \eqref{eq:e12var} with respect to $e_2$ and using \eqref{eq:alpha1alpha2}, we obtain
\begin{equation} \label{eq:d2e12}
\begin{array}{lll}
0 & = & -6c\nu_2K_1\nu_1-6c\nu(\nu_1K_{12}+\nu_2K_{22}+K_1\nu_{12})+2K_1K_{12} \\
& & +2c\nu\nu_2\Delta K-(K-c\nu^2)(\Delta K)_2+4(K-c)(2cK\nu\nu_2-8 c^2\nu^3\nu_2).
\end{array}
\end{equation}
After multiplication by $K- c \nu^2$ and reporting \eqref{E2-2}, this gives
$$0=(-6 c(K-4 c \nu^2)K_1\nu_1+B)\nu_2+(K- c \nu^2)(-6c\nu K_{12}\nu_1+C)$$ with
\begin{eqnarray*}
B & = & -3c\nu K_1^2+2c\nu(K-c\nu^2)(K_{11}-2K_{22}+4(K- c)(K-4 c \nu^2)),\\
C & = & 2K_1K_{12}-(K-c\nu^2)(\Delta K)_2.\\
\end{eqnarray*}
Multiplying by $\nu K_1$ and reporting \eqref{eq:e12var} yields
$$0=D\nu_2+(K-c\nu^2)\nu E$$
with
\begin{eqnarray*}
D & = & -(K-4 c \nu^2)K_1A+\nu K_1B, \\
E & = & -K_{12}A+K_1C \\
& = & K_1^2K_{12}+(K- c \nu^2)(K_{12}\Delta K-K_1(\Delta K)_2-4(K- c)(K+2 c \nu^2)K_{12}).
\end{eqnarray*}
Observe that $D$ factorizes as 
$$D=(K- c \nu^2)K_1F$$ with
$$F=K\Delta K-K_1^2-2c\nu^2K_{11}-8c\nu^2K_{22}-4K(K-c)(K-4c\nu^2).$$
From this we get 
\begin{equation} \label{eq:nu2}
0=K_1F\nu_2+\nu E.
\end{equation}
Finally, reporting \eqref{eq:e12var} and \eqref{eq:nu2} into \eqref{E1-1} gives
\begin{equation} \label{eq:orderzero}
A^2F^2+36 c^2\nu^4E^2+36 c^2K_1^2\nu^2(1-\nu^2)(K- c \nu^2)F^2=0.
\end{equation}

This equation holds on $U\cap V$ and extends by continuity to $U$. It has the desired form on $U$ but may not extend smoothly to $\Sigma$. So, observing that
$$K_1^2K_{11}=(\nabla^2K)(\nabla K,\nabla K),\quad
K_1^2K_{22}=(\nabla^2K)(J\nabla K,J\nabla K),$$
$$K_1^2K_{12}=(\nabla^2K)(\nabla K,J\nabla K),\quad
K_1(\Delta K)_2=\langle\nabla\Delta K,J\nabla K\rangle,$$
we see that if we multiply both sides of \eqref{eq:orderzero} by $K_1^4$ we obtain an equation of the form \eqref{eq:fourth}
that extends to the whole surface $\Sigma$, 
where, for each $x\in\Sigma$, $P(x,\cdot)$ is a polynomial map, and where the map $x\mapsto P(x,\cdot)$ is analytic. One can also easily check that $P(x,\cdot)$ is even and of degree at most $12$ for each point $x\in\Sigma$.

We now prove that $\Sigma\setminus Z$ is dense. Note that $\Sigma\setminus Z\subset U$. Assume that $\Sigma\setminus Z$ is not dense. Then $Z$ contains a non empty open set $\Omega$. Since $K$ is analytic and non constant, by taking a smaller set $\Omega$ if necessary, we may also assume that $K$, $K-c$ and $\nabla K$ do not vanish on $\Omega$.

The coefficients of orders $0$ and $12$ of $P(x,\cdot)$ vanish for every $x\in\Omega$. Using \eqref{eq:orderzero}, this leads to the following relations on $\Omega$:
\begin{equation} \label{eq:ozero}
 K\Delta K-K_1^2-4K^2(K- c)=0,
\end{equation}
\begin{equation} \label{eq:otwelve1}
K_{11}+4K_{22}-8K(K-c)=0, 
\end{equation}
\begin{equation} \label{eq:otwelve2}
 K_{12}=0.
\end{equation}

We are going to prove that these three equations lead to a contradiction. We recall that $K_2=0$. We deduce from \eqref{eq:alpha1alpha2} and \eqref{eq:otwelve2} that
\begin{equation} \label{eq:alpha1}
 \alpha_1=0.
\end{equation}
From \eqref{eq:ozero}, \eqref{eq:otwelve1} and \eqref{eq:alpha1alpha2} we get
\begin{equation} \label{eq:k11}
K_{11}=\frac83K(K-c)+\frac{4K_1^2}{3K}.
\end{equation}
\begin{equation} \label{eq:alpha2k1}
 \alpha_2K_1=K_{22}=\frac43K(K-c)-\frac{K_1^2}{3K}.
\end{equation}
On the other hand, by \eqref{eq:alpha1} we have
$$K=\langle\rmR(e_1,e_2)e_1,e_2\rangle=-\alpha_2^2-e_1\cdot\alpha_2,$$
so differentiating \eqref{eq:alpha2k1} with respect to $e_1$ yields
$$-(K+\alpha_2^2)K_1+\alpha_2K_{11}=\frac{8KK_1}3-\frac{4cK_1}3-\frac{2K_1K_{11}}{3K}+\frac{K_1^3}{3K^2}.$$
Reporting \eqref{eq:alpha2k1} and \eqref{eq:k11}, we obtain, after a straightforward computation, 
\begin{equation} \label{eq:KK1}
0=(K+20c)K_1^2-16K^2(K-c)^2. 
\end{equation}
Differentiating this equation with respect to $e_1$ and reporting \eqref{eq:k11}, we get, after another straightforward computation, 
\begin{equation} \label{eq:KK1bis}
0=(11K+160c)K_1^2-16K^2(11K^2-37cK+26c^2).
\end{equation}
It then follows from \eqref{eq:KK1} and \eqref{eq:KK1bis} that $K$ is a root of a non trivial polynomial. In particular, $K$ is constant  on $\Omega$, hence on $\Sigma$ by analyticity, which is a contradiction. This proves that $\Sigma\setminus Z$ is dense, and consequently that $P$ has the required properties. This concludes the proof in the case where $\Sigma$ is orientable.

We now assume that $\Sigma$ is not orientable. Let $\tilde\Sigma$ be the orientable double cover of $\Sigma$, with a given orientation. Then we can define $\tilde P:\tilde\Sigma\times\R\to\R$, and one can check that $\tilde P$ does not depend on the chosen orientation. Consequently we can define $P:\Sigma\times\R\to\R$ with the desired properties.
\end{proof}

\begin{cor} \label{numbernu}
 Assume that $K$ is not constant. Then there exist at most $12$  functions $\nu:\Sigma\to[-1,1]$ (and possibly none) satisfying \eqref{eq:gaussnu} and \eqref{eq:jacobinu}. Moreover, a function $\nu$ satisfies \eqref{eq:gaussnu} and \eqref{eq:jacobinu} if and only if $-\nu$ does.
\end{cor}

\begin{proof}
Let $\Omega\subset\Sigma\setminus Z$ be a non empty connected open set. Then any function $\nu:\Sigma\to[-1,1]$ satisfying \eqref{eq:gaussnu} and \eqref{eq:jacobinu} on $\Omega$ satisfies \eqref{eq:fourth} on $\Omega$. By the properties of $P$, for each $x\in\Omega$ the set $\{s\in[-1,1]\mid P(x,s)=0\}$ is finite and has cardinal at most $12$ (and is possibly empty). Hence, by analyticity, there are at most $12$ functions $\nu$ satisfying \eqref{eq:fourth} on $\Omega$. Then a function $\nu:\Sigma\to[-1,1]$ satisfying \eqref{eq:gaussnu} and \eqref{eq:jacobinu} on $\Sigma$ is necessarily the analytic continuation of one of these functions. The last assertion of the corollary is obvious.
\end{proof}

\begin{rem} \label{riccicondition}
Ricci proved that a simply connected Riemannian surface $\Sigma$ with metric $\rmd s^2$ having negative curvature $K$ can be minimally isometrically immersed into $\R^3$ if and only if the metric $\sqrt{-K}\rmd s^2$ is flat, i.e., 
\begin{equation} \label{eq:ricci}
 ||\nabla K||^2=K\Delta K-4K^3.
\end{equation}
This condition is called the \emph{Ricci condition} (see \cite{docarmodajczer} for a generalization to higher dimensions and all space forms). Recently, using the description of minimal surfaces in $\R^3$ in terms of meromorphic data and a study of log-harmonic functions, A. Moroianu and S. Moroianu \cite{moroianu} proved that a simply connected Riemannian surface can be minimally isometrically immersed into $\R^3$ if and only if its curvature $K$ satisfies $K\leqslant0$ and \eqref{eq:ricci}.

It is not clear whether such a simple necessary and sufficient condition holds for minimal isometric immersions into $\M^2(c)\times\R$ (we observe that setting $c=0$ in \eqref{eq:third} gives \eqref{eq:ricci}). We can differentiate \eqref{eq:e12var} with respect to $e_1$ or \eqref{eq:fourth} with respect to $e_1$ or $e_2$, and next report \eqref{eq:e12var} and \eqref{eq:nu2} to obtain other order zero equations satisfied by $\nu$, in terms of other derivatives of the curvature. However this seems to lead to very complicated computations.
\end{rem}

\section{Minimal surfaces with constant intrinsic curvature} \label{sec:constantK}

Trivial examples of minimal surfaces with constant intrinsic curvature are totally geodesic surfaces of $\M^2(c)\times\R$ (see Lemma \ref{constantangle}). We first describe a non trivial example, and then prove that all these minimal surfaces are either totally geodesic or congruent to a part of an associate surface of this example.

It is interesting to mention that constant intrinsic curvature surfaces (non necessarily minimal) in $\M^2(c)\times\R$ are studied in \cite{aeg}.

\begin{exa}[The parabolic generalized catenoid in $\h^2\times\R$] \rm \label{ex:catenoid}
We recall the following example from \cite{danieltams} when $c<0$. Here we do the normalization $c=-1$. Proposition 4.17 in \cite{danieltams} describes a properly embedded minimal surface $\cC$ in $\h^2\times\R$ having the following properties:
\begin{itemize}
 \item the intersection of $\cC$ with a horizontal plane $\h^2\times\{a\}$ is either empty or a horocycle,
\item all these horocycles have asymptotic points that project to the same point in $\partial_\infty\h^2$,
\item the surface $\cC$ is invariant by a one-parameter family of horizontal parabolic isometries,
\item the surface $\cC$ has a horizontal plane of symmetry.
\end{itemize}
We will call this surface a \emph{parabolic generalized catenoid}. Such a surface is unique up to isometries of $\h^2\times\R$. It belongs to a two-parameter families of minimal surfaces foliated by horizontal curves of constant curvature, which were classified by L. Hauswirth \cite{hauswirthpacific}. Moreover, this surface is a limit of rotational catenoids as the radius of the circle in their horizontal plane of symmetry tends to $+\infty$ (one has to fix a point of the circle and the tangent plane at this point to get the limit). The associate family of $\cC$ contains in particular the left and right helicoids with vertical period $\pi$ (see Proposition 4.20 in \cite{danieltams} and Example 18 in \cite{hset}).

The embedding given in Proposition 4.17 of \cite{danieltams} is an embedding $f:D\to\h^2\times\R$ where $D=(-\pi/2,\pi/2)\times\R$, and the induced metric is given in canonical coordinates $(u,v)$ by  $$\rmd s^2=\frac{\rmd u^2+\rmd v^2}{\cos^2u}.$$ One can check that this metric is complete and has constant curvature $-1$. Moreover, the angle function $\mu$ of the immersion is given by  $\mu=\sin u$. We consider the orthonormal frame $(e_1,e_2)$ defined by
$$e_1=\cos u\frac{\partial}{\partial u},\quad e_2=\cos u\frac{\partial}{\partial v}$$ and we use the notation of Section \ref{sec:resolution}. Then $\mu_1=\cos^2u$ and $\mu_2=0$.


We observe that the curve $\gamma$ of equation $v=0$ in $D$ is a geodesic. For $t\in\R$ we let $\varphi_t:D\to D$ denote the hyperbolic translation by $t$ along $\gamma$ with a chosen orientation. Then $f\circ\varphi_t:D\to\h^2\times\R$ is a minimal isometric immersion (actually an embedding) with angle function $\mu\circ\varphi_t$. When $w\neq t$, the functions $\mu\circ\varphi_w$ and  $\mu\circ\varphi_t$ are not equal, since $\mu$ is strictly monotonous along  $\gamma$. Consequently, the immersions $f\circ\varphi_w$ and $f\circ\varphi_t$ are not associate unless $w=t$. 

The same argument holds replacing hyperbolic translations along $\gamma$ by hyperbolic rotations around a given point. Hence we get an example of a Riemannian surface admitting an infinite number (actually a two-parameter family) of non associate minimal isometric immersions. 

We also observe that $\mu\circ\varphi_t\to1$ (respectively, $\mu\circ\varphi_t\to-1$) uniformly on compact sets as $t\to+\infty$ (respectively, $t\to-\infty$). We recall that the immersions given by Theorem \ref{thm:tams} are unique only up to isometries in $\isomzero(\h^2\times\R)$. So, if we furthermore fix  points $z_0\in D$ and $p_0\in\h^2\times\R$, then we can choose a smooth family $(\Psi_t)_{t\in\R}$ of isometries in $\isomzero(\h^2\times\R)$ such that $\Psi_t(f(\varphi_t(z_0)))=p_0$ for every $t$ and such that the immersion $\Psi_t\circ f\circ\varphi_t$ converges (uniformly on compact sets) as $t\to+\infty$ to a constant height immersion, that is, the corresponding limit surface is a horizontal hyperbolic plane.
\end{exa}


In the next theorem we classify constant intrinsic curvature minimal surfaces in $\M^2(c)\times\R$ (when $c>0$ this classification was obtained in \cite{tu2}).

\begin{thm} \label{thm:cc}
Let $\Sigma$ be a minimal surface in $\M^2( c)\times\R$ with constant intrinsic curvature $K$. Then
\begin{itemize}
 \item either $\Sigma$ is totally geodesic and $K=0$ or $K= c$,
\item either $ c<0$, $K= c$ and $\Sigma$ is part of an associate surface of the parabolic generalized catenoid.
\end{itemize}
\end{thm}

\begin{proof}
Let $\nu:\Sigma\to[-1,1]$ denote the angle function of $\Sigma$. We note that, since $K$ is constant, \eqref{eq:gaussnu} and \eqref{eq:jacobinu} imply that the function $\nu$ is \emph{isoparametric}, that is, $\nabla\nu$ and $\Delta\nu$ are functions of $\nu$ only (this is similar to the situations treated in \cite{km,tutams,tu2}). Since $K$ is constant, \eqref{eq:third} becomes
$$0=4(K-c)(K-c\nu^2)(K+2c\nu^2).$$

If $K\neq c$, then this implies that $K=c\nu^2$ or $K=-2c\nu^2$. In particular $\nu$ is constant, so Lemma \ref{constantangle} gives the result.

Assume now that $K=c$. If $\nu$ is constant, then Lemma \ref{constantangle} gives the result. So we now assume that $\nu$ is not a constant function. By analyticity we may restrict ourselves to a simply connected open set $\Omega\subset\Sigma$ on which $\nabla\nu$ does not vanish. Then \eqref{eq:gaussnu} implies that $c(1-\nu^2)^2<0$, so $c<0$.

Up to scaling, we may assume that $c=-1$. We set $D=(-\pi/2,\pi/2)\times\R$ and we endow $D$ with the metric $$\rmd s^2=\frac{\rmd u^2+\rmd v^2}{\cos^2u}.$$ We saw in Example \ref{ex:catenoid} that this metric has curvature $-1$ and is complete. Hence, we can assume that $\Omega$ is given by an immersion $U\to\M^2( c)\times\R$ for some open domain $U\subset D$. We consider the orthonormal frame $(e_1,e_2)$ defined by
$$e_1=\cos u\frac{\partial}{\partial u},\quad e_2=\cos u\frac{\partial}{\partial v}$$ and we use the notation of Section \ref{sec:resolution}. 


By applying isometries of $D$, we may assume that $(0,0)\in U$, $\nu_1(0,0)>0$ and $\nu_2(0,0)=0$. We set $a=\nu(0,0)$. Then by \eqref{E1-1} we have $\nu_1(0,0)=1-a^2$; in particular $a\in(-1,1)$. Then $\nu$ is solution to \eqref{E1-1}, \eqref{E2-1}, \eqref{E2-2}, \eqref{E2-3} with initial conditions
$$\nu(0,0)=a,\quad \nu_1(0,0)=1-a^2,\quad \nu_2(0,0)=0.$$ 

For $(u,v)\in D$ we set $\mu(u,v)=\sin u$. By the discussion in Example \ref{ex:catenoid}, $\mu$ is the angle function of the parabolic generalized catenoid and for every $t\in\R$  the function $\mu\circ\varphi_t$ is a solution to \eqref{E1-1}, \eqref{E2-1}, \eqref{E2-2}, \eqref{E2-3}, where $\varphi_t$ is as in Example \ref{ex:catenoid}.

We now claim that there exists $t\in\R$ such that 
$$(\mu\circ\varphi_t)(0,0)=a,\quad(\mu\circ\varphi_t)_1(0,0)=b,\quad(\mu\circ\varphi_t)_2(0,0)=0.$$ 

Indeed, since $a\in(-1,1)$, there exists $u_0\in(-\pi/2,\pi/2)$ such that $\mu(u_0,0)=a$. Let $t\in\R$ be such that $\varphi_t(0,0)=(u_0,0)$. Then clearly $(\mu\circ\varphi_t)(0,0)=a$. Also for $i=1,2$ we have $\rmd_{(0,0)}\varphi_t(e_i(0,0))=e_i(\varphi_t(0,0))$, so $(\mu\circ\varphi_t)_1(0,0)=\mu_1(\varphi_t(0,0))=1-a^2$ and $(\mu\circ\varphi_t)_2(0,0)=\mu_2(\varphi_t(0,0))=0$. This proves the claim.

Since there is at most one function satisfying \eqref{E1-1}, \eqref{E2-1}, \eqref{E2-2}, \eqref{E2-3} with those intial conditions, we obtain that $\nu=\mu\circ\varphi_t$. Then by Theorem \ref{thm:system} and by analyticity $\Sigma$ is  part of an associate surface of the parabolic generalized catenoid.
\end{proof}

We mention that there is an important literature about minimal isometric immersions of space forms into space forms or complex space forms in arbitrary dimensions (see for instance \cite{docarmowallach,bryanttams,km,lizhenqi} and references therein). Note that $\h^2\times\R$ embeds isometrically as a totally geodesic hypersurface in $\h^2\times\h^2$. It is perhaps interesting to try to generalize Theorem \ref{thm:cc} for immersions into higher dimensional products of space forms.


\section{On the set of minimal isometric immersions of a given surface} \label{sec:number}

\begin{thm} \label{thm:number}
 Let $\Sigma$ be a simply connected Riemannian surface with non constant curvature. Then, up to congruences, the set of minimal isometric immersions from $\Sigma$ to $\M^2(c)\times\R$ is empty or consists of at most six families of associate immersions.
\end{thm}

\begin{proof}
 By Theorem \ref{thm:system}, the angle function $\nu:\Sigma\to[-1,1]$ of such an immersion satisfies \eqref{eq:gaussnu} and \eqref{eq:jacobinu}, and is not constant by Lemma \ref{constantangle}. By Corollary \ref{numbernu} there exist at most $12$ such functions $\nu$. Again by Theorem \ref{thm:system} each of these functions $\nu$ gives, up to congruences, at most one family of associate immersions. We conclude using Remark \ref{minusnu}.
\end{proof}

\begin{cor} \label{family}
Let $\Sigma$ be a simply connected Riemannian surface with non constant curvature. Let $(f_t)_{t\in I}$ be a continuous family of minimal isometric immersions from $\Sigma$ to $\M^2(c)\times\R$, where $I$ is a real interval. Then all immersions $f_t$, $t\in I$, are associate.
\end{cor}

\begin{proof}
Let $Y$ be the set of smooth functions $\nu:\Sigma\to[-1,1]$ satisfying \eqref{eq:gaussnu} and \eqref{eq:jacobinu}. For each $t\in I$, let $\nu_t:\Sigma\to[-1,1]$ be the angle function of $f_t$. Then the map $t\mapsto\nu_t$ is a continuous map from $I$ to $Y$, which is finite by the proof of Theorem \ref{thm:number}. So this map is constant, that is, all immersions $f_t$ have the same angle function, hence they are associate by Theorem \ref{thm:system}.
\end{proof}

\begin{rem} \label{rem:family}
By Theorem \ref{thm:cc}, the hypothesis in Theorem \ref{thm:number} and Corollary \ref{family} that the surface has non constant curvature can be removed if $c>0$ and replaced by the hypothesis that the surface does not have constant curvature $c$ when $c<0$.
\end{rem}

\begin{rem}
Let $\Sigma$ be a simply connected Riemannian surface and $f:\Sigma\to\M^2(c)\times\R$ a minimal isometric immersion. Let $\sigma\in\isom(\M^2(c)\times\R)$ be the reflection with respect to a horizontal totally geodesic surface and $\varphi\in\isom(\M^2(c)\times\R)$  the reflection with respect to a vertical totally geodesic surface. We assume that the curvature of $\Sigma$ is not constant. Then there exists a continuous one-parameter family of minimal isometric immersions of $\Sigma$ containing both $f$ and $\sigma\circ f$ (namely the associate family of $f$, since $\sigma\circ f=f^\pi$), while there does not exist a continuous one-parameter family of minimal isometric immersions of $\Sigma$ containing both $f$ and $\varphi\circ f$ (by Corollary \ref{family} and Lemma \ref{constantangle}, since $f$ and $\varphi\circ f$ have opposite angle functions; see Proposition 3.8 in \cite{danieltams}).
\end{rem}

The following example, due to R. Sa Earp \cite{saearp}, is an example of non constant curvature Riemannian surfaces admitting two non associate minimal isometric immersions into $\h^2\times\R$ up to congruences. 

\begin{ex} \label{ex:saearp}
R. Sa Earp \cite{saearp} classified minimal surfaces in $\h^2\times\R$ that are invariant by a one-parameter family of hyperbolic screw motions. They constitute a two-parameter family $(\cS_{\ell,d})_{(\ell,d)\in\R^2}$, the parameters being the pitch (or slope) $\ell$ of the screw motion and the prime integral $d$ of the ordinary differential equation defining the generatrix curve. 

Note that the statements in \cite{saearp} asserting that there exists a two-parameter family of minimal surfaces that are isometric to a given surface $\cS_{\ell,d}$ are misleading, since changing the extra parameter $m$ in that paper simply corresponds to multiplying the second coordinate on the surface by a constant. Indeed, if $f_{\ell,d,m}$ denotes the immersion in \cite{saearp} with parameters  $\ell$, $d$, $m$, then, in the notation of  \cite{saearp}, $f_{\ell,d,m}(s,\tau)=f_{\ell,d,1}(s,\tau/m)$ (this can be seen from Theorem 4.2 in \cite{saearp}). In general, the immersions $f_{\ell,d,m}$ and $f_{\ell,d,1}$ do not induce the same metrics, since the map $(s,\tau)\mapsto(s,\tau/m)$ is not an isometry (by formula (15) in \cite{saearp}).

When $\ell=0$, the surface is invariant by a one-parameter family of horizontal hyperbolic isometries (it belongs to the family of L. Hauswirth \cite{hauswirthpacific} and it is also described in \cite{danieltams}, p. 6279). When $d=0$, the generatrix curve is a horizontal geodesic (it also belongs to the family of L. Hauswirth \cite{hauswirthpacific}). 

R. Sa Earp observed that the surfaces $\cS_{\ell,d}$ and $\cS_{\underline\ell,\underline d}$ are isometric but not associate if $d^2>1$, $\underline d^2<1$ and $(d^2-1)/(\ell^2+1)=(1-\underline d^2)/(\underline d^2+\underline\ell^2)$. This provides examples of Riemannian surfaces admitting two minimal isometric immersions into $\h^2\times\R$ that are not associate up to congruences. A model of such a surface is $\R^2$ with the metric is $$\rmd s^2=\rmd u^2+\Lambda(u)^2\rmd v^2$$
where $$\Lambda(u)=\sqrt{(d^2-1)\cosh^2u+\ell^2+1}.$$ We consider the orthonormal frame $(e_1,e_2)$ defined by
$$e_1=\frac\partial{\partial u},\quad e_2=\frac1{\Lambda(u)}\frac\partial{\partial v}$$ and we use the notation of Section \ref{sec:resolution}. Then $$[e_1,e_2]=-\frac{\Lambda'(u)}{\Lambda(u)}e_2,\quad\alpha_1=0,\quad\alpha_2=\frac{\Lambda'(u)}{\Lambda(u)},$$
so $$K=\langle\rmR(e_1,e_2)e_1,e_2\rangle=-\frac{\Lambda''(u)}{\Lambda(u)}
=-1+\frac{(\ell^2+1)(d^2+\ell^2)}{\Lambda(u)^4}.$$ 
We can check the functions $\nu$ and $\underline\nu$ such that
$$\nu^2(u)=\frac{(d^2-1)\cosh^2u}{(d^2-1)\cosh^2u+\ell^2+1}=1-\frac{\ell^2+1}{\Lambda(u)^2},$$
$$\underline\nu^2(u)=\frac{(d^2-1)\sinh^2u}{(d^2-1)\cosh^2u+\ell^2+1}=1-\frac{d^2+\ell^2}{\Lambda(u)^2}$$ 
satisfy \eqref{eq:gaussnu} and \eqref{eq:jacobinu}. They are the angle functions of $\cS_{\ell,d}$ and $\cS_{\underline\ell,\underline d}$ respectively.

The surfaces $\cS_{\ell,d}$ and $\cS_{0,\delta}$ with $\delta=\sqrt{1+\frac{d^2-1}{\ell^2+1}}$ are associate (this can be seen making the change of coordinate $v\mapsto\sqrt{\ell^2+1}v$). 
\end{ex}




\begin{ex}
Using a computer algebra system, we can check that a minimal surface in $\h^2\times\R$ (respectively, in $\s^2\times\R$) that is locally isometric to a catenoid (respectively, to an unduloid) is one of its associate surfaces (see \cite{hauswirthpacific,danieltams} for descriptions of catenoids and unduloids). 

Indeed, the metrics are given in an open subset of $\R^2$ by $\rmd s^2=\rmd u^2+\Lambda(u)^2\rmd v^2$ with $\Lambda(u)=\sqrt{\beta^2\sinh^2u+1}$ for some $\beta\in(-\infty,-1)\cup(1,+\infty)$ (respectively $\Lambda(u)=\sqrt{\beta^2\sin^2u+1}$ for some $\beta\neq0$). Using the orthonormal frame as that of Example \ref{ex:saearp}, we get $\alpha_1=0$, $K_2=0$, $K_{12}=0$ and $(\Delta K)_2=0$ (since $\Delta K=K_{11}+K_{22}=e_1\cdot K_1+\alpha_2K_1$ is a function of $u$ only). Hence $E=0$, so \eqref{eq:nu2} gives $\nu_2=0$ or $F=0$. If $F=0$, then $\nu$ is a function of $u$ only (since the coefficients in $F$ are functions of $u$ only and are not identically $0$), so we also have $\nu_2=0$.

Then reporting \eqref{eq:e12var} into \eqref{eq:gaussnu} gives a pointwise polynomial equation for $\nu$ of degree $8$. We obtain up to a sign four complex valued solutions: two of them do not satisfy \eqref{eq:gaussnu}, another one does not take values in $[-1,1]$. Hence there is, up to a sign, a unique solution to \eqref{eq:gaussnu} and \eqref{eq:jacobinu} taking values in $[-1,1]$.

In a similar way we obtain that the surfaces in Example \ref{ex:saearp} do not have non associate isometric minimal surfaces other than those described in that example (up to congruences).
\end{ex}

\begin{rem}
We do not know any example of a Riemannian surface admitting non associate minimal isometric immersions into $\s^2\times\R$ (up to congruences).

We do not know any example of a non constant curvature Riemannian surface admitting more than two non associate minimal isometric immersions into $\h^2\times\R$ (up to congruences).
\end{rem}

\begin{rem}
Let $H\neq0$. The problem of classifying non congruent CMC $H$ isometric immersions of a Riemannian surface into $\s^2\times\R$ or $\h^2\times\R$ remains largely open. F. Torralbo and F. Urbano \cite{tutams}  related this problem to a question about surfaces with parallel mean curvature vector in $\s^2\times\s^2$ or $\h^2\times\h^2$ that are not minimal. They provided some examples and classified pairs of non congruent CMC $H$ isometric immersions of a Riemannian surface having the same angle function. 

This problem is quite different from the case of minimal immersions. For instance, it is not known whether simply connected CMC $H$ surfaces admit a one-parameter family of CMC $H$ isometric deformations. On the other hand, CMC surfaces in $\s^2\times\R$ and $\h^2\times\R$ are related to minimal surfaces in simply connected homogeneous Riemannian $3$-manifolds with a $4$-dimensional isometry group by a local isometric Lawson-type correspondence \cite{danielcmh}. 

It is also interesting to mention that J. G\'alvez, A. Mart\'{\i}nez and P. Mira \cite{gmmcag} answered the question whether a surface in these homogeneous $3$-manifolds is uniquely dermined by its metric and its principal curvatures.
\end{rem}


%
%

\bibliographystyle{plain}
\bibliography{isometric}

\end{document}